\documentclass[11pt,reqno]{amsart}
\usepackage{amssymb}
\usepackage{fullpage}
\usepackage{amsthm}
\usepackage{bm}
\usepackage{latexsym}
\usepackage{float}
\restylefloat{table}
\usepackage{amsmath}
\usepackage{eufrak}
\usepackage{mathrsfs}
\usepackage[font={small,it}]{caption}
\usepackage{amscd}
\usepackage[all,cmtip]{xy}
\usepackage[usenames]{color}
\usepackage{graphicx}
\usepackage{color}
\usepackage{amscd}
\usepackage{float}
\usepackage{graphics}
\usepackage{tikz}
\usepackage{tikz-cd}
\usepackage{comment}
\usepackage{xspace}
\usepackage{mathtools}

\usepackage{amssymb}
\usepackage{amsthm}
\usepackage{graphicx}
\usepackage{subfig}
\usepackage{enumerate}
\usepackage{hyperref}

\usetikzlibrary{patterns,decorations.pathreplacing}
\usepackage{marginnote}

\theoremstyle{plain}

\newtheorem{theorem}{Theorem}[section]
\newtheorem{proposition}[theorem]{Proposition}

\newtheorem{lemma}[theorem]{Lemma}
\newtheorem{corollary}[theorem]{Corollary}

\DeclareFontFamily{U}{wncy}{}
    \DeclareFontShape{U}{wncy}{m}{n}{<->wncyr10}{}
    \DeclareSymbolFont{mcy}{U}{wncy}{m}{n}
    \DeclareMathSymbol{\Sh}{\mathord}{mcy}{"58}
\theoremstyle{definition}

\newcommand{\appsection}[1]{\let\oldthesection\thesection
\renewcommand{\thesection}{Appendix \oldthesection}
\section{#1}\let\thesection\oldthesection}

\theoremstyle{remark}

\newtheorem{remark}[theorem]{Remark}

\DeclareMathOperator{\spec}{Spec}

\DeclareMathOperator{\Sp}{sp}
\DeclareMathOperator{\Sel}{Sel}
\DeclareMathOperator{\Ch}{char}

\def\R{{\mathbb{R}}}
\def\Z{{\mathbb{Z}}}
\def\F{{\mathbb{F}}}
\def\Q{{\mathbb{Q}}}
\def\C{{\mathbb{C}}}
\def\P{{\mathbb{P}}}

\def\frakd{{\mathfrak{d}}}

\def\frakp{{\mathfrak{p}}}
\def\frakn{{\mathfrak{n}}}
\def\frakm{{\mathfrak{m}}}

\DeclareMathOperator{\rk}{rank}
\DeclareMathOperator{\ord}{ord}

\def\calT{{\mathcal{T}}}

\def\WW{{\mathcal{W}}}

\usepackage{microtype}

\DeclareMathOperator{\Sym}{Sym}
\DeclareMathOperator{\Gal}{Gal}

\begin{document}
\title{Watkins' conjecture for elliptic curves over function fields}
\author[Jerson Caro]{Jerson Caro}
\email{jocaro@uc.cl}
\address{Facultad de Matem\'aticas, Pontificia Universidad Cat\'olica de Chile, Campus San Joaqu\'in, Avenida Vicu\~na Mackenna 4860, Santiago, Chile.}


\maketitle
\begin{abstract}
In 2002 Watkins conjectured that given an elliptic curve defined over $\Q$, its Mordell-Weil rank is at most the $2$-adic valuation of its modular degree. We consider the analogous problem over function fields of positive characteristic, and we prove it in several cases. More precisely, every modular semi-stable elliptic curve over $\mathbb{F}_q(T)$ after extending constant scalars, and every quadratic twist of a modular elliptic curve over $\mathbb{F}_q(T)$ by a polynomial with sufficiently many prime factors satisfy the analogue of Watkins' conjecture. Furthermore, for a well-known family of elliptic curves with unbounded rank due to Ulmer, we prove the analogue of Watkins' conjecture. 
\end{abstract}
\section{Introduction} \label{intro}
Let $\mathcal{E}$ be an elliptic curve over $\Q$ of conductor $N$. The modular degree $m_{\mathcal{E}}$ of $\mathcal{E}$ is the minimum degree of all modular parametrizations $\phi:X_0(N)\to \mathcal{E}$ over $\Q$. The modularity Theorem \cite{Wiles1995,taylor1995,breuil2001} implies that it is well-defined. In 2002 Watkins \cite{Watkins2002} conjectured that for every elliptic curve $\mathcal{E}$ over $\Q$ we have $r \leq \nu_2(m_{\mathcal{E}})$, where $\nu_2$ denotes the $2$-adic valuation and $r\coloneqq\rk_\Z (\mathcal{E}(\Q))$.

Let $k$ be a finite field of characteristic $p>3$, write $A=k[T]$ for the polynomial ring, and let $K=k(T)$ be its fraction field. Let $\infty$ denote the place of $K$ associated with $1/T$. Let $E$ be a non-isotrivial (see Section 2.3 for the definition) elliptic curve defined over $K$. Under the assumption that $E$ has split multiplicative reduction at $\infty$, there is an analogue to the modularity Theorem cf. Theorem \ref{modularity theorem}. Namely, if $E$ is non-isotrivial and has split multiplicative reduction at $\infty$ and conductor ideal $\mathfrak{n}$, then there is a non-constant map $\phi_E:X_0(\mathfrak{n})\to E$, where $X_0(\mathfrak{n})$ is the corresponding Drinfeld modular curve. Thus, from now on we say that $E$ is modular if it is non-isotrivial and has split multiplicative reduction at $\infty$. Given a modular elliptic curve $E$ over $K$, we say that it satisfies Watkins' conjecture if $\rk_\Z (E(K))\leq \nu_2(m_E)$, where $m_E$ is the minimal degree of a modular parametrization $\phi_E$.

Using Atkin-Lehner involutions  we prove a potential version of Watkins' conjecture for semi-stable elliptic curves over $K$ (see \cite{dummigan2013powers} and \cite{caro2021} for other applications of Atkin-Lehner involutions in the context of Watkins' conjecture).

\begin{theorem}\label{Wconjecture extension}
Let $E$ be a modular semi-stable elliptic curve defined over $K$ with conductor $\frakn_E=(n)\infty$. Let $k'$ be a finite field containing the splitting field of $n$ over $k$, then Watkins' conjecture holds for $E'=E\times_{\spec K}\spec K'$, where $K'\coloneqq k'(T)$. 
\end{theorem}

It is not known whether the Mordell-Weil rank of elliptic curves over $\Q$ is unbounded or not. Over $K$ we know that the rank is unbounded thanks to the work of Shafarevitch and Tate \cite{tate67} in the isotrivial case and Ulmer \cite{ulmer2002elliptic} in the non-isotrivial case. The next result proves Watkins' conjecture for one of the families given by Ulmer, thus, we obtain Watkins' conjecture for elliptic curves over $K$ with arbitrarily large rank.

\begin{theorem}\label{Wconjecture Ulmer}
Let $p$ be a prime and $n$ be a positive integer, such that $6\mid p^n+1$. The elliptic curve
\[
E:y^2+T^dxy=x^3-1
\]
where $d=(p^n+1)/6$ defined over $\F_q(T)$,
satisfies Watkins' conjecture.
\end{theorem}

On the other hand, Esparza-Lozano and Pasten \cite{Esparza2021} prove that, over $\Q$, the quadratic twist $\mathcal{E}^{(D)}$ of $\mathcal{E}$ by $D$ satisfies Watkins' conjecture whenever the number of distinct prime divisor of $D$ is big enough. Using results of Papikian \cite{papikian2007analogue} on $L(\Sym^2f,2)$ over function fields, when $f$ is a Drinfeld modular form, we can prove an analogue over function fields. In the following we write $\omega_K(g)$ for the number of distinct irreducible factors of a polynomial $g$ in $A$.

\begin{theorem}\label{Wconjecture twist} 
Let $E$ be an elliptic curve over $K$ with minimal conductor among its quadratic twists. Let its conductor be $\frakn\infty=(n_1^2n_2)\infty$, where $n_1, n_2$ are square-free coprime polynomials. Assume that $E$ has a non-trivial $K$-rational $2$-torsion. Let $g$ be a monic square-free polynomial of even degree such that $gcd(n_1,g)=1$, and $\omega_K(g)\geq 2\omega_K(\frakn)-\nu_2(m_E)$, then Watkins' conjecture holds for $E^{(g)}$.
\end{theorem}

The condition that $g$ has even degree is necessary to guarantee that $E^{(g)}$ is modular (cf. Section \ref{Section Twists}). The previous Theorem will be used to deduce the following: 

\begin{corollary}\label{Wconjecture twists ss}
Assume that $E$ is a semi-stable modular elliptic curve over $K$. Then we have that $E^{(g)}$ satisfies Watkins' conjecture whenever $\omega_K(g)\geq 3$. Furthermore, if every prime dividing $\frakn$ has non-split multiplicative reduction and $E(K)[2]\cong \Z/2\Z$ then $E^{(g)}$ satisfies Watkins' conjecture for every square-free polynomial $g\in A$ of even degree.
\end{corollary}
\section{Preliminaries}
The idea of this section is to define the associated invariants to Watkins' conjecture over function fields. Write $K_\infty$ for the completion of $K$ at $T^{-1}$, and let $\mathcal{O}_{\infty }$ be its ring of integers. Let $\C_{\infty }$ denote the completion of an algebraic closure of $K_\infty$.
\subsection{Drinfeld Modular Curves}
We denote by $\Omega$ the Drinfeld upper half plane $\C_\infty- K_\infty$. Notice that $GL(2, K_\infty)$ acts on $\Omega$ by fractional linear transformations, in particular, so does the Hecke congruence subgroup associated with an ideal $\frakn$ of $A$
\[
\Gamma_0(\mathfrak{n})=\left\{g=\begin{pmatrix}a & b\\
c & d\end{pmatrix}\in G\colon a,b,c,d\in \F_q[T],\,c\equiv 0\text{ (mod }\mathfrak{n}),\, \det (g)\in \mathcal{O}_{\infty}\right\}.
\]
The compactification of the quotient space $\Gamma_0(\frakn)\backslash\Omega$ by the finitely many cusps $\Gamma_0(\frakn)\backslash\P^1(K)$ is the Drinfeld modular curve. We denoted it by $X_0(\frakn)$. 
\subsection{Drinfeld Modular Forms and Hecke Operators}
In this section, we define an analogue of the cuspidal Hecke newforms over $\C$. Another way to understand $\Omega$ is the Bruhat-Tits tree $\calT$ of $PGL(2,K_\infty)$, whose oriented edges are in correspondence with the cosets of $GL(2, K_\infty)/K_{\infty}^{\times}\cdot \mathcal{J}$ (see Section 4.2 \cite{Gekeler1996}), where
\[
\mathcal{J}=\left\{\begin{pmatrix}a&b\\
c &d\end{pmatrix}\in GL(2,O_\infty)\colon c\equiv 0\text{ (mod }T^{-1})\right\}.
\] 
This correspondence gives an action of $GL(2,K_\infty)$ on the real-valued functions on the oriented edges of $\calT$ by left-multiplying the argument. Let $\underline{H}_{!}(\Gamma_0(\frakn),\R)$ be the finite-dimensional $\R$-space of real-valued, alternating, harmonic and $\Gamma_0(\frakn)$-invariant functions on the oriented edges of $\calT$ having finite support modulo $\Gamma_0(\frakn)$.

For each divisor $\frakd=(d)$ of $\frakn$, let $i_\frakd$ be the map
\[
i_\frakd\colon (\underline{H}_{!}(\Gamma_0(\frakn/\frakd),\R))^2\longrightarrow \underline{H}_{!}(\Gamma_0(\frakn),\R)
\]
given by 
\[
i_{\frakd}(f,g)(e)=f(e)+g\left(\begin{pmatrix}d&0\\0&1\end{pmatrix}\cdot e\right),
\]
for every oriented edge $e$. The subspace of oldforms at level $\frakn$  is
\[
\underline{H}^{old}_{!}(\Gamma_0(\frakn),\R)=\sum_{\frakp\mid\frakn}i_{\frakp}((\underline{H}_{!}(\Gamma_0(\frakn/\frakp),\R))^2).
\]
Denote by $\underline{H}^{new}_{!}(\Gamma_0(\frakn),\R)$ to the orthogonal complement of the oldforms with respect to the Petersson-norm (see Section 4.8 Gekeler \textit{op. cit.}) defined over $\underline{H}_{!}(\Gamma_0(\frakn),\R)$.

For any nonzero ideal $\frakm$ there is a Hecke operator $T_\frakm$, for example, for $\frakm$ relatively prime to $\frakn$ is defined by
\[
T_{\frakm}f(e)=\sum f\left(\begin{pmatrix}a&b\\0&d\end{pmatrix}\cdot e\right),
\]
where the sum runs over $a, b, d\in A$ such that $a,d$ are monic, $\frakm=(ad)$, and $\deg(b)<\deg(d)$, see Section 4.9 Gekeler \textit{op. cit.} for a general definition. Finally, a newform is a normalized Drinfeld modular form $f\in \underline{H}^{new}_{!}(\Gamma_0(\frakn),\R)$, and an eigenform for all Hecke operators.
\subsection{Elliptic curves}
Let $E$ be an elliptic curve defined over $K$. Assume that $E$ has an affine model
\begin{equation}\label{EqnM1}
Y^2 + a_1XY+a_3Y = X^3+a_2X^2+a_4X+a_6.
\end{equation}
where $a_i\in K$. For this cubic equation, define the usual Weierstrass invariants:
\begin{align*}
b_2&=a_1^2+4a_2,\quad b_4=a_1a_3+2a_4, \quad b_6=a_3^2+4a_6,\\
b_8&=a_1^2a_6-a_1a_3a_4+4a_2a_6+a_2a_3^2-a_4^2,\\
c_4&=b_2^2-24b_4,\quad c_6=-b_2^3+36b_2b_4-216b_6,\\
\Delta&=-b_2^2b_8-8b_4^3-27b_6^2+9b_2b_4b_6,\\
j_E&=c_4^3\Delta^{-1}.
\end{align*}

We say that $E$ is \textit{non-isotrivial} when $j_E\notin k$. Since we assume that $\Ch(k)>3$ the conductor of $E$ is cubefree. Denote it by $\frakn_E$ and by $\frakn$ its finite part, in particular, $\frakn_E=\frakn\cdot \infty^{i}$, where $i\in\{0,1,2\}$. When $E$ has split multiplicative reduction at $\infty$, due to Drinfeld's reciprocity law (Proposition 10.3 \cite{Drinfeld1974}) and the fact that $E$ is automorphic (Theorem 9.8 in \cite{Deligne1973}), there is an analogue of the modularity Theorem over $\Q$:
\begin{theorem}[Modularity Theorem]\label{modularity theorem}
Let $E$ be an elliptic curve over $K$ of conductor $\frakn_E=\frakn_0\cdot\infty$ having split multiplicative reduction at $\infty$. There is a non-constant morphism $X_0(\frakn)\to E$ defined over $K$.
\end{theorem}
\begin{remark}
This Theorem gives a bijection between primitive newforms $f$ (i.e., $f$ is a newform such that $f\notin n\underline{H}^{new}_{!}(\Gamma_0(\frakn),\Z)$ for $n > 1$) with integer eigenvalues and isogeny classes of modular elliptic curves over $K$ with conductor $\frakn\cdot \infty$.
\end{remark}
\subsubsection{$L$-functions}
There is an attached $L$-function to an elliptic curve with conductor $\frakn_E$, which has an Euler product expansion
\[
L(E,s)=\sum_{n\text{ pos. div.}}\frac{a_n}{|n|^{s}}=\prod_{\frakp}\left(1-\frac{\alpha_\frakp}{|\frakp|^s}\right)^{-1}\left(1-\frac{\beta_\frakp}{|\frakp|^s}\right)^{-1},
\]
where $\alpha_{\frakp}, \beta_{\frakp}$ are defined as follows: (1) if $\frakp\nmid \frakn_E$, $\alpha_{\frakp}+ \beta_{\frakp}=a_\frakp\coloneqq|\frakp|+1-\#E(\F_\frakp)$ and $\alpha_{\frakp} \beta_{\frakp}=|\frakp|$, (2) if $\frakp \mid\mid \frakn_E$, $\alpha_{\frakp}=0$ and $\beta_{\frakp}=\pm1$, and (3) if $\frakp^2 \mid \frakn_E$, $\alpha_{\frakp}= \beta_{\frakp}= 0$. 

Due to results of Grothendieck \cite{Grothendieck1964} and Deligne \cite{Deligne1973} $L(E,s)=L(f_E,s)$, where $f_E$ is the newform associated to $E$, and $L(E,s)$ is a polynomial in the variable $q^{-s}$ of degree $\deg(\frakn)-4$.

Over this newform $f_E$ we define the $L$-function attached to its symmetric square $L(\Sym^2f_E,s)$ with the following local factors
\[
L_{\mathfrak{p}}(\Sym^2f_E,s)=\begin{cases}
1,& \text{if }\mathfrak{p}^2\mid \frakn_E,\\
\left(1-\frac{1}{|\mathfrak{p}|^s}\right)^{-1}, & \text{if }\mathfrak{p}\mid\mid \frakn_E,\\
\left(1-\frac{\alpha_{\mathfrak{p}}^2}{|\mathfrak{p}|^s}\right)^{-1}
\left(1-\frac{\alpha_{\mathfrak{p}}\overline{\alpha_{\mathfrak{p}}}}{|\mathfrak{p}|^s}\right)^{-1}\left(1-\frac{\overline{\alpha_{\mathfrak{p}}}^2}{|\mathfrak{p}|^s}\right)^{-1}& \text{if }\mathfrak{p}\nmid \frakn_E.
\end{cases} 
\]
When $E$ is semi-stable Proposition 5.4 from \cite{papikian2002degree} implies that $L(\Sym^2f_E,s)$ is a polynomial in the variable $q^{-s}$ of degree $2\deg(\frakn_E) -4$.

\subsubsection{Upper Bounds for the Rank of the Mordell-Weil Group}
The following is a geometric bound for the Mordell-Weil rank due to Tate \cite{tate65}
\begin{equation}\label{geometric bound}
\rk_\Z(E(K))\leq \ord_{s=1}L(E,s)\leq \deg(\frakn_E)-4.
\end{equation}
See \cite{Ulmer2011} for detailed proof. In addition, if the elliptic curve $E$ has a non-trivial $K$-rational $2$-torsion, we can give an upper bound for its Mordell-Weil rank in terms of $\omega_K(\frakn)$, the number of distinct primes that divide $\frakn$ in $A$.

First of all, notice that the change of variables  $X=z/4$, $Y= y/8 - a_1z/8 -a_3/2$ transforms \eqref{EqnM1} into
\begin{equation}\label{EqnM2}
y^2 = z^3 + b_2z^2+8b_4z + 16b_6.
\end{equation}
Let $\gamma\in K$ be a root of the previous cubic, associated to a non-trivial $K$-rational $2$-torsion point.  Then $\gamma\in A$  and the change of variables $z=x+\gamma$ turns \eqref{EqnM2} into
\begin{equation}\label{EqnM3}
y^2 = x^3+Ax^2+Bx
\end{equation}
where 
$$
A= 3\gamma + b_2 \quad \mbox{and}\quad B=3\gamma^2 + 2b_2\gamma+8b_4.
$$
 Let $\Delta^{min}$ be the discriminant of the minimal model \eqref{EqnM1} and let $\Delta$ be the discriminant of \eqref{EqnM3}. Notice that $\Delta=2^{12}\Delta^{\min}$ by the standard transformation formulas, thus, \eqref{EqnM3} is a minimal model of $E$. Now, recall the usual exact
sequence related to a $2$-descent,
\begin{equation}\label{1-sequence}
\xymatrix{0\ar[r]& \frac{E(K)}{2E(K)}\ar[r]& \Sel_2(E/K)\ar[r]&\Sh(E/K)[2]\ar[r]&0}.
\end{equation}
Furthermore, consider the exact sequence from Lemma 6.1 of \cite{schaefer2004}
\begin{equation}\label{2-sequence}
\xymatrix{
  0 \ar[r] &\frac{E'(K)[\theta']}{\phi(E(K)[2])} \ar[r] & \Sel^{\theta}(E/K) \ar[r] & \Sel_{2}(E/K) \ar[r]& \Sel^{\theta'}(E'/K).
}    
\end{equation}
These two exact sequences imply that $\rk_\Z(E(K))+2\leq s(E,\theta)+s'(E,\theta)$, where $s(E,\theta)=\dim_{\F_2}(\Sel^{\theta}(E/K))$ and $s'(E,\theta)=\dim_{\F_2}(\Sel^{\theta'}(E'/K))$. In addition, there is a correspondence between Selmer groups and homogeneous spaces (see Chapter 4 from \cite{roberts2007explicit}), which shows that $s(E,\theta)\leq \omega_K(A^2-4B)+1$ and $s'(E,\theta)\leq\omega_K(B)+1$. Thus, we have the following proposition:

\begin{proposition}\label{rank bound}
Let $E$ be an elliptic curve with $K$-rational $2$-torsion and Weierstrass minimal model $y^2=x^3+Ax^2+Bx$, then:
\[
\rk_\Z(E(K))  \leq \omega_K(A^2 - 4B) + \omega_K(B),
\]
consequently, if $\alpha$ (resp. $\mu$) is the number of primes of additive (resp. multiplicative) bad reduction of $E/K$. Then:
\[
\rk_\Z(E(K)) \leq \mu+2\alpha.
\]
\end{proposition}

\subsubsection{Modular Degree}
Let $E$ be a modular elliptic curve defined over $K$. Let $X_0(\frakn)$ be the Drinfeld modular curve parametrizing $\phi_E\colon X_0(\frakn)\to E$
where $\phi_E$ is non-trivial and of minimal possible degree. \textit{The modular degree} $m_E$ is the degree of $\phi_E$. The following Lemma relates the $2$-adic valuations of $m_E$ and $L(\Sym^2f, 2)$.
\begin{lemma}\label{2-adic degree}
Let $E$ be a modular elliptic curve with conductor $\frakn\infty$. Then we have that
\[
\nu_2(m_E)=\nu_2(L(\Sym^2f, 2))-\nu_2(val_{\infty}(j_E)).
\]
\end{lemma}
\begin{proof}
Proposition 1.3 in \cite{papikian2007analogue} states that
\[
m_E=\frac{q^{\deg\frakn-2}(\widetilde{c_E})^2}{-val_\infty(j_E)}L(\Sym^2f, 2),
\]
where $\widetilde{c}_E$ is the Manin constant and $q=\#k$. By taking $2$-adic valuations we obtain
\[
\nu_2(m_E)=\nu_2(q^{\deg\frakn-2}(\widetilde{c}_E)^2)+\nu_2(L(\Sym^2f, 2))-\nu_2(val_{\infty}(j_E)),
\]
and by Proposition 1.2 from \cite{Pal2010} $\widetilde{c_E}$ is a power of $q$ which yields the desired result.
\end{proof}

\section{Watkins' Conjecture for Semi-stable Elliptic Curves}
For any ideal $\frakm=(m)$, such that $\frakm\mid\frakn=(n)$, and $\frakm$ and $\frakn/\frakm$ are relatively primes, there is an Atkin-Lehner involution $W_{\frakm}$. 
This involution acts on $\underline{H}_{!}(\Gamma_0(\frakn),\R)$ as follows
\[
W_{\frakm}f(e)=f\left(\begin{pmatrix}ma & b\\ nc & md \end{pmatrix}\cdot e\right),
\]
where $a, b,c, d\in A$ and $m^2ab-nbc=\gamma m$ for some $\gamma\in k^{\times}$. We denote by $\WW(\frakn)$ the $2$-elementary abelian group of all Atkin-Lehner involutions. Let $f$ be a primitive newform; since $f$ is primitive, it is determined by its eigenvalues up to sign. By Lemma 11 from \cite{Atkin1970} the Hecke operators commute with the Atkin-Lehner involutions, hence $W_{\mathfrak{p}}^{(\mathfrak{n})}f$ and $f$ have the same Hecke eigenvalues. By Lemma  1.2 from \cite{schweizer 98} $\underline{H}^{new}_{!}(\Gamma_0(\frakn),\R)$ is stable under the Atkin-Lehner involutions, and consequently, we have that $W_{\mathfrak{p}}f=\pm f$.
\begin{remark}\label{Atkin eigenform}
Let $E$ be a modular elliptic curve, and $f_E$ be its attached primitive newform, then $f_E$ is an eigenform of every Atkin-Lehner involution.
\end{remark}
The following Proposition gives a lower bound of $\nu_2(m_E)$ in terms of $\omega_K(\frakn)$.
\begin{proposition}\label{power 2 modular}
Let $E$ be an elliptic curve with conductor $\frakn_E=\mathfrak{n}\infty$. Let $f_E$ be the primitive newform associated to $E$. Over this newform, we define $\WW'=\{W\in \mathcal{W}\colon W(f_E)=f_E\}$, and 
$\kappa\coloneqq\dim_{\F_2}([\mathcal{W}(\frakn):\mathcal{W}'])+\dim_{\F_2}(E(K)[2])$.
Then $\omega_K(\mathfrak{n})-\kappa\leq\nu_2( m_E)$.
\end{proposition}
\begin{proof}
Proposition 10.3 from \cite{Drinfeld1974} gives the following isomorphism 
\[
H^{1}(X_0(\frakn)\otimes K_{\infty}^{sep},\Q_\ell)\cong \underline{H}_{!}(\Gamma_0(\frakn),\Q_\ell)\otimes \Sp,
\]
where $\Sp$ is the two-dimensional special $\ell$-adic representation of $\Gal(K_{\infty}^{sep}/K_{\infty})$. Furthermore, this isomorphism
is compatible with the action of the Atkin-Lehner involutions. 

Since $H^{1}(X_0(\frakn)\otimes K_{\infty}^{sep},\Q_\ell)$ is the dual of $V_\ell(J_0(\frakn))$, we have that if $\pi:J_0(\frakn)\to E$ is the projection, then $\pi([W(D)])=\pi([D])$ for every divisor $D$ of degree $0$ over $X_0(\frakn)$ whenever $W\in \mathcal{W}'$.
By Remark \ref{Atkin eigenform} $\mathcal{W}'$ has at most index $2$ in $\mathcal{W}(\frakn)$. Now, as in Proposition 2.1 in \cite{dummigan2013powers} we construct a homomorphism $\theta:\mathcal{W}'\to E(K)[2]$. First of all, we fix a $K$-rational point $x_0\in X_0(\frakn)$, then for $W\in\mathcal{W}'$ we define $\theta(W)=\pi([W(x_0)-(x_0)])$. Notice that $\theta(W)\in E(K)[2]$, since $x_0\in X_0(\frakn)(K)$ and 
\[
\theta(W)=\pi([W(x_0)-(x_0)])=\pi([W(W(x_0)-(x_0))])=-\pi([W(x_0)-(x_0)])=-\theta(W).
\]
Now, define $\mathcal{W}''=\ker \theta$. Let $\mathcal{X}=X_0(\frakn)/\mathcal{W}''$, and denote by $\psi:X_0(\frakn)\to \mathcal{X}$ that is also defined over $K$ and by $\mathcal{J}$ the Jacobian of $\mathcal{X}$. We can define $\iota:X_0(\frakn)\to J_0(\frakn)$ based on $x_0$, and $\iota':\mathcal{X}\to \mathcal{J}$ based on $\psi(x_0)$, so we obtain a commutative diagram
\[
\xymatrix{
X_0(\frakn)\ar[rr]^{\iota}\ar[d]_{\psi}& 
&J_0(\frakn)\ar[d]^{\psi_{*}}\\
\mathcal{X}\ar[rr]^{\iota'}& 
&\mathcal{J}.}
\]
Since $\pi([W(x_0)-x_0])=0$ for $W\in\mathcal{W}''$, we have that $\pi\circ \iota(w(x))=\pi\circ \iota(x)$ for all $x\in X_0(\frakn)$, in particular, $\pi\circ \iota$ factors through $\mathcal{X}$. Since the image of $\iota$ generates to $J_0(\frakn)$ as a group, there exists $\pi': \mathcal{J}\to E$ such that $\pi=\pi'\circ \psi_{*}$, then
\[
[m_E]=\pi\circ \pi^{\vee}=(\pi'\circ \psi_{*})\circ (\psi^{*}\circ \pi'^{\vee})=\pi'\circ [\deg(\psi)]\circ \pi'^{\vee}=[\#\mathcal{W}'']\circ(\pi'\circ \pi'^{\vee}).
\]
Since the degree of $[i]$ (multiplication by $i$) is $i\cdot i^{*}$ or $(i^{*})^2$, where $i^{*}$ denotes the $p$-free part of $i$, then $\#\mathcal{W}''\mid m_E$, since $p\neq2$.
\end{proof}
The previous Proposition and Tate's geometric bound \eqref{geometric bound} allow us to prove Theorem \ref{Wconjecture extension}.  

\begin{proof}[Proof of Theorem \ref{Wconjecture extension}]
Recall that $E'=E\times_{\spec K}\spec K'$. Since the conductor of $E'$ is also $\frakn_E=(n)\infty$, then by Tate's geometric bound \eqref{geometric bound} $\rk(E'(K'))\leq \deg(n)-4$. On the other hand, we know that $\omega_{K'}((n))=\deg(n)$ because $k'$ contains the splitting field of $n$. Furthermore, since $\dim_{\F_2}([\mathcal{W}(\frakn):\mathcal{W}'])\leq 1$, by Remark \ref{Atkin eigenform}, we have $\kappa\leq 3$, then by Proposition \ref{power 2 modular} we have that
\begin{align*}
\nu_2(m_{E'})\geq \omega_K((n))-3=\deg(n)-3=\deg(\frakn_E)-4\geq \rk(E'(k'(T))),
\end{align*}
which yields the desired result.
\end{proof}
Ulmer \cite{ulmer2002elliptic} exhibits a closed formula for the rank of a family of elliptic curves. Proposition \ref{power 2 modular} together with this formula allow us to show Watkins' conjecture for this family.
\begin{proof}[Proof of Theorem \ref{Wconjecture Ulmer}]
First of all, we notice that $E(\overline{\F_p}(T))[2]=(0)$, since the polynomial $4x^3+T^{2d}x-4$ does not have solution over $\overline{\F_p}(T)$.
Notice that $E$ is the change of base point of $\P^1$ given by $[0:1]\mapsto \infty$ of 
\[
E':y^2+xy=x^3-T^{m},
\]
where $m=p^n+1$. Theorem 1.5 in \cite{ulmer2002elliptic} shows that $\frakn_{E'}=T(1-2^43^3T^m)$, then in particular $\frakn_{E}=(T^m-2^43^3)\infty$.
We claim that $f(T)=T^m-2^43^3$ always has a root in $\F_{p^2}$. Let $\alpha\in\F_{p^2}$ such that $\alpha^2=3$, and notice that if $\alpha\in \F_p$, $2^23\alpha$ is a root of $f$. If $\alpha\notin\F_p$, since $6\mid p^n-1$ we have that $p\equiv -1 \text{(mod }3)$, then $p\equiv 1\text{(mod }4)$ by the law of quadratic reciprocity. This implies that $2^23\alpha$ or $2^23\alpha\beta$ is a root of $f$, where $\beta^2=-1$. Consequently, there is a bijection between the prime divisors of even degree of $T^m-1$ and $f(T)$. 

By definition, $T^m-1$ factors over $\F_p[T]$ as follows 
\[
T^m-1=\prod_{e\mid m}\Phi_e(T),
\]
where $\Phi_n(T)$ is the $n^{th}$-cyclotomic polynomial. Thus, the number of prime divisors over $\F_q[T]$ of $f(T)$ is 
\begin{equation*}
\omega_{\F_q(T)}(\frakn_E)=\sum_{e\mid m}\frac{\phi(e)}{o_e(q)}-\begin{cases}
0&\text{if $T^m-2^43^3$ has solution in $\F_q$}\\
1&\text{otherwise}\end{cases},
\end{equation*}
where $\phi(e)$ is the cardinality of $(\Z/e\Z)^{\times}$ and $o_e(q)$ is the order of $q$ in $(\Z/e\Z)^{\times}$. On the other hand, we know that $\rk(E(\F_p(T)))=\rk(E'(\F_p(T)))$. Theorem 1.5 in \cite{ulmer2002elliptic} states a closed expression for $\rk(E'(\F_q(T)))$
\begin{equation*}
\sum_{\substack{e\mid m\\e\nmid 6}}\frac{\phi(e)}{o_e(q)}+\begin{cases}
2&\text{if $3\mid q-1$}\\
1 &\text{otherwise}
\end{cases}+\begin{cases}
1&\text{if $4\mid q-1$}\\
0 &\text{otherwise}
\end{cases}.
\end{equation*}
Since there are $4$ divisors of $6$ we obtain
\[
\sum_{e\mid m}\frac{\phi(e)}{o_e(q)}\geq\sum_{\substack{e\mid m\\e\nmid 6}}\frac{\phi(e)}{o_e(q)}+4
\]
Furthermore, if $3\mid q-1$ then $q$ is a square since $p\equiv -1 \text{(mod }3)$; which implies that $T^m-2^43^3$ has solution in $\F_q$. 
Hence, Proposition \ref{power 2 modular} implies that
\begin{align*}
\nu_2(m_E)\geq\omega_{\F_q(T)}(\frakn_E)-1=\sum_{e\mid m}\frac{\phi(e)}{o_e(q)}-1\geq\sum_{\substack{e\mid m\\e\nmid 6}}\frac{\phi(e)}{o_e(q)}+3\geq \rk(E(\F_q(T))).
\end{align*}
Finally, if $3\nmid q-1$ we obtain
\begin{align*}
\nu_2(m_E)\geq\omega_{\F_q(T)}(\frakn_E)-1\geq\sum_{e\mid m}\frac{\phi(e)}{o_e(q)}-2\geq\sum_{\substack{e\mid m\\e\nmid 6}}\frac{\phi(e)}{o_e(q)}+2\geq \rk(E(\F_q(T))),
\end{align*}
which gives the desired result.
\end{proof}

\section{Watkins' Conjecture for Quadratic Twists}\label{Section Twists}

Let $E$ be a modular elliptic curve with conductor $\frakn_E$, since $\Ch(k)>3$ there exist square-free coprime polynomials $n_1,n_2\in A$ such that $\frakn_E=(n_1^2n_2)\infty$. For $g\in A$ be a monic square-free polynomial, with $(n_1,g)=1$, we define the quadratic twist $E^{(g)}$ of $E$ by $g$ as follows
\[
E^{(g)}\colon y^2 = x^3+Agx^2+Bg^2x.
\]
We assume that $\deg(g)$ is even to ensure that $E^{(g)}$ is modular. To see that, notice that if the change of variables $x\mapsto T^{2n}x$ and $y\mapsto T^{3n}y$ makes $E$ a minimal $T^{-1}$-integral model, then the change $x\mapsto T^{2(n+m)}x$ and $y\mapsto T^{3(n+m)}y$ makes $E^{(g)}$ a minimal $T^{-1}$-integral model, where $\deg(g)=2m$; since $g$ is a monic polynomial, both reductions modulo $T^{-1}$ are the same. Note that the conductor $\frakn_E^{(g)}$ of $E^{(g)}$ is equal to $\frakn_E (g^2/d)$, where $d=\text{gcd}(n_2,g)$. We denote by $f^{(g)}$ to the associated Drinfeld newform to $E^{(g)}$.

The following lemma gives an upper bound for the Mordell-Weil rank of $E^{(g)}$.
\begin{lemma}\label{rank twist}
With the notation above, we have that
\[
\rk_\Z(E^{(g)}(K))\leq \omega_K(n_2)+2(\omega_K(n_1)+\omega_K(g)).
\]
\end{lemma}
\begin{proof}
First of all, we notice that $E^{(g)}$ has multiplicative reduction at $\frakp$ if $\frakp\mid n_2/d$, $E^{(g)}$ has additive reduction at $\frakp$ if $\frakp\mid n_1g$, and otherwise $E^{(g)}$ has good reduction at $\frakp$. Then by Proposition \ref{rank bound}
we obtain that 
\[
\rk_\Z(E^{(g)}(K))\leq\omega_K(n_2/d)+2( \omega_K(n_1)+\omega_k(g)),
\]
since $\omega_K(n_2/d)\geq \omega_K(n_2)$ we obtain the desired result.
\end{proof}

To find a lower bound for $\nu_2(m_{E^{(g)}})$, we need to relate $L(\Sym^2f^{(g)},2)$ and $L(\Sym^2f,2)$, so, we can use Lemma \ref{2-adic degree} and the fact that $j_E=j_{E^{(g)}}$ (since this two elliptic curves are isomorphic in a quadratic extension of $K$), but before, we need the following lemma
\begin{lemma}\label{twist a_p}Let $\mathfrak{p}$ be a prime ideal of $A$ and let $\left(\frac{\cdot}{\mathfrak{p}}\right)\colon \F_{\mathfrak{p}}\to \{-1,0, 1\}$ be the extended Legendre symbol. Then 
\[
a_{\mathfrak{p}}(E^{(g)})=\left(\frac{g}{\mathfrak{p}}\right)a_{\mathfrak{p}}(E).
\]
\end{lemma}
\begin{proof}
If $E^{(g)}$ has additive reduction at $\frakp$, we have that $\frakp\mid n_1$ or $\frakp\mid g$, then $a_{\mathfrak{p}}(E^{(g)})=0$ and there is nothing to prove. On the other hand, assume that $E^{(g)}$ has multiplicative reduction at $\frakp$. By Lemma 2.2 in \cite{conrad2005root} $E$ has split multiplicative reduction at $\frakp$ if and only if $\left(\frac{-c_6(E)}{\mathfrak{p}}\right)=1$, as a consequence, this quantity is equal to $a_{\mathfrak{p}}(E)$. Furthermore, since $c_6(E^{(g)})=g^3c_6(E)$, we have
\[
a_{\mathfrak{p}}(E^{(g)})=\left(\frac{-c_6(E^{(g)})}{\mathfrak{p}}\right)=\left(\frac{-g^3c_6(E)}{\mathfrak{p}}\right)=\left(\frac{g}{\mathfrak{p}}\right)a_{\mathfrak{p}}(E).
\]
Finally, assume that $\frakp\nmid\frakn^{(g)}$, Define $M=\{x\in \F_\frakp\colon x^3+Ax^2+B\neq 0\}$. Consequently, we obtain
\begin{align*}
\#E^{(g)}_{\frakp}(\F_{\frakp})&=|\frakp|+1+\sum_{x\in M}\left(\frac{x^3+Agx^2+Bg^2x}{\mathfrak{p}}\right)\\
&=|\frakp|+1+\sum_{x\in M}\left(\frac{g^3(x^3+Ax^2+Bx)}{\mathfrak{p}}\right)\\
&=|\frakp|+1+\left(\frac{g}{\mathfrak{p}}\right)\sum_{x\in M}\left(\frac{x^3+Ax^2+Bx}{\mathfrak{p}}\right)\\
&=|\frakp|+1-\left(\frac{g}{\mathfrak{p}}\right)a_{\mathfrak{p}}(E^{(g)}),
\end{align*}
by recalling the definition of $a_{\frakp}(E)$ we get the desired result.
\end{proof}
\begin{proposition}\label{L function twist}
Let $E$ be a modular elliptic curve with conductor $\frakn_E$ and associated primitive newform $f$. Assume that $E'$ is a quadratic twist of $E$, with conductor $\frakn'_E$ and associated primitive newform $f'$, such that $\ord_\frakp(\frakn_E)\leq \ord_\frakp(\frakn'_E)$ for all $\frakp$. Thus, there exist $n_1,n_2,d,g$ square-free monic polynomials with $1=\gcd(n_1,g)$, and $d=\gcd(n_2,g)$ such that $\frakn_E=(n_1^2n_2)\infty$ and $\frakn'_E=\frakn_E g^2/d$. Then one has
\begin{align*}
L(\Sym^2f',2)=L(\Sym^2f,2)\frac{|d|}{|g|^3}\prod_{\frakp\mid d}(|\frakp|^2-1)\prod_{\frakp\mid  g/d}\left((|\frakp|+1)^2-a_{\mathfrak{p}}(E)^2\right)(|\frakp|-1).
\end{align*}
\end{proposition}
\begin{proof}
By Lemma \ref{twist a_p} we have that when $\ord_{\frakp}(\frakn)=\ord_{\frakp}(\frakn')$ the local factors are equal, i.e. $L_{\frakp}(\Sym^2f',2)=L_{\frakp}(\Sym^2f,2)$. If $\frakp\mid d$, we have that
\[
L_{\frakp}(\Sym^2f',s)=L_{\frakp}(\Sym^2f,s)(1-|\frakp|^{-s}),
\]
thus, at $s=2$ we obtain
\[
L_{\frakp}(\Sym^2f',2)=L_{\frakp}(\Sym^2f,2)\frac{1}{|\frakp|^2}(|\frakp|^2-1).
\]
Finally, assume that $\frakp\mid (g/d)$. The local factors are related as follows
\[
L_{\frakp}(\Sym^2f',s)=L_{\frakp}(\Sym^2f,s)\left(1-\alpha_{\mathfrak{p}}^2|\mathfrak{p}|^{-s}\right)
\left(1-\overline{\alpha_{\mathfrak{p}}}^2|\mathfrak{p}|^{-s}\right)\left(1-|\mathfrak{p}|^{1-s}\right),
\]
therefore at $s=2$ we obtain
\[
L_{\frakp}(\Sym^2f',2)=L_{\frakp}(\Sym^2f,2)\frac{1}{|\frakp|^3}\left((|\frakp|+1)^2-a_{\mathfrak{p}}(E)^2\right)(|\frakp|-1),
\]
putting all together, we achieve the desired result.
\end{proof}

\begin{proof}[Proof of Theorem \ref{Wconjecture twist}]
Since $E$ and $E^{(g)}$ are isomorphic over $\C_\infty$, we have that $j_E=j_{E^{(g)}}$, thus by Lemma \ref{2-adic degree} we obtain 
\[
\nu_2(m_{E^{(g)}})=\nu_2(m_E)+\nu_2(L(\Sym^2f^{(g)}, 2))-\nu_2(L(\Sym^2f, 2)).
\]
On the other hand, Proposition \ref{L function twist} implies that
\[
\nu_2(L(\Sym^2f^{(g)}, 2)/L(\Sym^2f, 2))=\sum_{\frakp\mid d}\nu_2(|\frakp|^2-1)+\sum_{\frakp\mid  g/d}\nu_2\left(((|\frakp|+1)^2-a_{\mathfrak{p}}(E)^2)(|\frakp|-1)\right).
\]
We know that $|\frakp|^2-1\equiv 0 \text{ (mod }8)$, meanwile $|\frakp|-1\equiv 0 \text{ (mod }2)$. As $E(K)[2]$ is non-trivial and it maps injectively into $E_\frakp(\F_\frakp)$ for every prime $\frakp\nmid \frakn\infty$, then $|\frakp|+1-a_{\mathfrak{p}}(E)\equiv 0 \text{ (mod }2)$, which implies $(|\frakp|+1)^2-a_{\mathfrak{p}}(E)^2\equiv 0 \text{ (mod }4)$. As a consequence
\[
\nu_2(L(\Sym^2f^{(g)}, 2))-\nu_2(L(\Sym^2f, 2))\geq 3\omega_K(g).
\]
Putting all together, we achieve the result. 
\begin{equation}\label{equ modular}
\nu_2(m_{E^{(g)}})\geq \nu_2(m_E)+3\omega_K(g). 
\end{equation}
By Proposition \ref{rank bound} we know that $\rk(E^{(g)})\leq 2(\omega_K(\frakn)+\omega_K(g))$. By our assumptions on $g$ we obtain that
\begin{equation*}
\nu_2(m_E)+3\omega_K(g)\geq 2(\omega_K(\frakn)+\omega_K(g)),    
\end{equation*}
consequently, $\rk(E^{(g)})\leq\nu_2(m_{E^{(g)}})$.
\end{proof}
\begin{proof}[Proof of Corollary \ref{Wconjecture twists ss}]
By Proposition \ref{power 2 modular} we have that $\nu_2(m_E)\geq \omega_K(\frakn)-3$. Since $E$ is semi-stable, $\frakn$ is square-free, consequently, Lemma \ref{rank twist} implies that $\rk(E^{(g)})\leq \omega_K(\frakn)+2\omega_K(g)$. Using the equation \eqref{equ modular}, we have
\begin{equation}\label{equ Watkins}
\nu_2(m_{E^{(g)}})\geq \nu_2(m_E)+3\omega_K(g)\geq \omega_K(\frakn)-3+3\omega_K(g)\geq \omega_K(g)-3+\rk(E^{(g)}),    
\end{equation}
hence Watkins' conjecture holds for $E^{(g)}$, whenever $\omega_K(d)\geq 3$. Furthermore, if a prime ideal $\frakp$ divides $\frakn$ and has non-split multiplicative reduction, by Theorem 3 in \cite{Atkin1970} $W_\frakp f=f$, consequently, $\mathcal{W}=\mathcal{W}'$. Therefore, if every prime $\frakp$ which divides $\frakn$ has non-split multiplicative reduction and $E(K)[2]\cong \Z/2\Z$ Proposition \ref{power 2 modular} implies that $\nu_2(m_E)\geq \omega_K(\frakn)-1$, thus, equation \eqref{equ Watkins} turns into
\[
\nu_2(m_{E^{(g)}})\geq \omega_K(g)-1+\rk(E^{(g)}),
\]
accordingly, Watkins' Conjecture holds for every square-free polynomial $g$ of even degree.
\end{proof}

\section*{Acknowledgements}
I want to thank Professor Hector Pasten for suggesting this
problem to me and for numerous helpful remarks. I was supported by ANID Doctorado Nacional 21190304.

\end{document}